\documentclass{article}

\parskip \smallskipamount

\usepackage{hyperref}
\hypersetup{
    colorlinks=true,
    linktocpage=true,
    breaklinks=true,
    urlcolor= blue,
    linkcolor= red,
    citecolor=blue,
}
\usepackage{amsmath,amsthm,amssymb}

\usepackage{tikz}
\usepackage{enumerate}
\usetikzlibrary{shapes}
\usetikzlibrary{decorations.pathreplacing}
\usetikzlibrary{patterns}

\newtheorem{lemma}{Lemma}
\newtheorem{theorem}{Theorem}
\newtheorem{proposition}{Proposition}

\newtheorem{definition}{Definition}
\newcommand{\N}{\mathbb{N}}
\newcommand{\R}{\mathbb{R}}
\newcommand{\Z}{\mathbb{Z}}
\newcommand{\E}{\mathbb{E}}
\newcommand{\PP}{\mathbb{P}}
\newcommand{\1}{\mathbf{1}}
\newcommand{\cA}{\mathcal{A}}
\newcommand{\cB}{\mathcal{B}}
\newcommand{\cC}{\mathcal{C}}
\newcommand{\s}{\mathfrak{s}}
\newcommand{\Stab}{\mathrm{Stab}}
\newcommand{\rc}{\rho_c}
\newcommand{\rs}{\rho_\star}

\title{A new proof of superadditivity and of the density conjecture for Activated Random Walks on the line}

\author{Nicolas Forien\thanks{
  CEREMADE, CNRS, Université Paris-Dauphine, Université PSL,
75016 Paris, France}}

\begin{document}

\maketitle

\begin{abstract}
In two recent works, Hoffman, Johnson and Junge
proved the density conjecture, the hockey stick conjecture and the
ball conjecture for Activated Random Walks in
dimension one, showing an equality between several different
definitions of the critical density of the model.
This establishes a kind of self-organized criticality, which was
originally predicted for the Abelian Sandpile Model.

Their proof uses a comparison with a
percolation process, which exhibits
superadditivity.
We present here a different proof of these conjectures, based on a new
superadditivity property that we establish directly for Activated Random
Walks, without relying on a percolation process.

This more elementary approach yields less precise bounds than the
percolation
technology developed by Hoffman, 
Johnson and Junge, but it might open new
perspectives to go beyond the one-dimensional
setting.
\end{abstract}

\section{Introduction}

We start by presenting the model and the main conjectures for which we give new proofs.
We refer to~\cite{LS24} for a more general presentation of these
conjectures (where they correspond to Conjectures~1,~11,~12 and~17),
along with other nice predictions on the model.

\subsection{Presentation of the model}
\label{sec-presentation}

Activated Random Walks is a system of interacting particles which is
defined as follows.
A configuration of the model on a graph consists of a
certain number of particles on each vertex, each particle being either
active or sleeping.
Active particles perform independent continuous-time random walks with
jump rate~$1$, according to a certain jump kernel on
the graph.
When an active particle is alone on a site, it falls asleep at a
certain rate~$\lambda>0$.
A sleeping particle stops moving and is instantaneously reactivated
as soon as it shares its site with at least one other particle.
If reactivated, the particle resumes its continuous-time random walk.

We say that the system fixates if every site of the graph is visited
only finitely many times, and otherwise we say that the system stays
active.
The model on~$\Z^d$, for every~$d\geq 1$, undergoes the following
phase transition:

\begin{theorem}
\label{thm-phase-transition}
In any dimension~$d\geq 1$, for every sleep
rate~$\lambda>0$ and every translation-invariant jump
kernel on~$\Z^d$ which
generates all~$\Z^d$, there exists~$\rc\in(0,1)$ such that, for every
translation-ergodic initial distribution with no sleeping particles
and an average density of active particles~$\rho$, the Activated
Random Walk model on~$\Z^d$ with sleep rate~$\lambda$ almost surely
fixates if~$\rho<\rc$, whereas it almost surely stays active
if~$\rho>\rc$.
\end{theorem}

This result is due to Rolla, Sidoravicius and
Zindy~\cite{RS12,RSZ19} who showed the existence
of the threshold density and its dependence on the mean density of particles
only, and to a series of
works~\cite{RS12,ST17,ST18,Taggi19,BGH18,HRR20,FG22,Hu22,AFG22}
which established the non-triviality of the
threshold, that is to say, that it is strictly between~$0$ and~$1$.
Note that, unlike some other models, the behaviour of Activated Random
Walks is not trivial even in one dimension.

A configuration of the model on a graph~$(V,E)$ can be represented by a
vector~$\eta:V\to\N\cup\{s\}$, where~$\eta(x)=k\in\N$ means that
there are~$k$ active particles at~$x$ and~$\eta(x)=\s$ means that
there is one sleeping particle at~$x$.
A site~$x$ is called unstable in~$\eta$ if it hosts at least one
active particle, and otherwise it is called stable.
We say that a configuration~$\eta$ is stable in some subset~$U\subset
V$ if~$U$ does not contain any active particle, i.e.,
if~$\eta(x)\in\{0,\,\s\}$ for every~$x\in U$.
If~$U\subset V$ we denote by~$\|\eta\|_U$ the total number of
particles (active or sleeping) in~$U$ in the configuration~$\eta$.

In the remainder of the paper, we restrict the presentation to the
one-dimensional case and we fix once and for all a sleep
rate~$\lambda>0$ and a translation-invariant
nearest-neighbour jump kernel on~$\Z$, which simply boils down to the
choice of a probability to jump
to the left, denoted by~$p\in(0,1)$.
All the statements hold for every~$\lambda>0$ and for
every~$p\in(0,1)$.

\subsection{The driven-dissipative Markov chain}
\label{sec-driven-dissipative}

The density conjecture, which is the content of
Theorem~\ref{thm-dd} below, connects the phase transition described
above with another version of the model called the driven-dissipative
system.

Fix an integer~$n\geq 1$, and consider the
segment~$V_n=\{1,\,\ldots,\,n\}\subset\Z$.
The driven-dissipative system consists of a Markov chain on the
finite set~$\{0,\,\s\}^{V_n}$ of
all stable configurations on~$V_n$.
At each time step of the Markov chain, an active particle is added to
a site of~$V_n$ chosen uniformly at random, and we let the system
evolve, with particles being killed when they jump out of~$V_n$ (by
the left or right exit), until a new stable configuration is reached.
This new stable configuration
gives the state of the Markov chain at the following time step.
This defines an irreducible and aperiodic Markov chain, which is
called driven-dissipative because the
system is driven by addition of active particles and there is
dissipation at the borders of~$V_n$.
We denote by~$S_n$ the number of sleeping
particles in a configuration sampled from the stationary distribution
of this Markov chain.

\subsection{Density conjecture}
\label{sec-statement-dd}

The density conjecture states that, when the length of the segment
tends to infinity, the density of particles~$S_n/n$ in this stationary
distribution concentrates around the critical density of
Theorem~\ref{thm-phase-transition}.
More precisely, we show the following result:

\begin{theorem}
\label{thm-dd}
We have~$S_n/n\to\rc$ in probability as~$n\to\infty$.
Moreover, for every~$\rho<\rho_c$ there exists~$c>0$ such that for
every~$n\geq 1$,~$\PP(S_n\leq\rho n)\leq e^{-cn}$.
\end{theorem}

In~\cite{HJJ1} a stronger result is obtained, with an exponential
bound also on the upper tail of~$S_n/n$ (see their Proposition~8.6),
and an interesting consequence is established in~\cite{BHS24}, namely
that the model on~$\Z$ with supercritical density but only one
active particle remains active with positive probability.

\subsection{Hockey stick conjecture}

While Theorem~\ref{thm-dd} shows a convergence for the stationary
measure of the driven-dissipative Markov chain, the hockey stick
conjecture complements this information with a prediction about the
behaviour of this Markov chain at all times.
Fix an integer~$n\geq 1$.
For every~$t\in\N$, we denote by~$Y_t$ the number of particles
remaining in~$V_n$
after~$t$ steps of the driven-dissipative Markov chain on~$V_n$, started
with the empty configuration at time~$t=0$.
The hockey stick conjecture consists in the following result:

\begin{theorem}
\label{thm-hockey}
For every~$\rho>0$ we have~$Y_{\lceil\rho
n\rceil}/n\to\min(\rho,\,\rc)$ in probability when~$n\to\infty$.
\end{theorem}

The denomination hockey stick refers to the shape of the curve of the
function~$\rho\mapsto\min(\rho,\,\rc)$.
This result, also conjectured more generally in~\cite{LS24}, was first
proved in~\cite{HJJ2} using the technology developed by~\cite{HJJ1}.

\subsection{Ball conjecture}

Another interesting setting consists of~$k$ active particles started
at the origin of the line~$\Z$, with no particles on the
other sites of~$\Z$.
We then denote by~$A_k$ the random set of sites which are visited at
least once when the system evolves from this initial configuration.
This set is called the aggregate, and we show that the final density
of sleeping particles
particles in the aggregate
concentrates around the critical density:

\begin{theorem}
\label{thm-ps}
We have~$k/|A_k|\to \rc$ in probability as~$k\to\infty$.
\end{theorem}

Propositions~8.7 and~8.8 in~\cite{HJJ1} make this result more precise
by showing that with probability at least~$1-e^{-ck}$ the aggregate
contains a segment and is contained in a slightly larger segment, both
centered on the origin.
Note that in our setting the limit of the aggregate is not necessarily
centered on the origin
because we also include the case of biased walks (for which the
technology of~\cite{HJJ1} is expected to extend).

This property is known as the ball conjecture because in higher
dimensions it is conjectured that the limiting shape of the
aggregate is a Euclidean ball centered on the origin, with a
density~$\rc$ of sleeping particles inside~\cite{LS24}.

\subsection{Content of this paper}

This paper presents a new proof, still in the one-dimensional case, of the
three conjectures presented above.
These conjectures were already proved by Hoffman, Johnson and Junge
in~\cite{HJJ1,HJJ2}, but using a different technique based on a
percolation process.
They use a superadditivity property of the percolation process to
obtain concentration bounds on~$S_n$.

The main added value of the present paper consists in a new
superadditivity property directly established for~$S_n$, without
relying on a percolation process.
The key consequence of this
superadditivity is that~$S_n/n$ converges to a constant
value~$\rs$, with an exponential bound on the lower tail.
The superadditivity property and this corollary are presented in
Section~\ref{sec-ingredients}, along with several earlier results on
which we rely, and some open questions.
Their proofs occupy Sections~\ref{sec-proof-superadd} and~\ref{sec-cv}.

Then, to obtain Theorem~\ref{thm-dd} (the density conjecture), there
only remains to show that this limit~$\rs$ is equal to the critical density~$\rc$ of the
model on~$\Z^d$, defined in Section~\ref{sec-presentation}.
This is done in Section~\ref{sec-dd}.

Finally, in Sections~\ref{sec-hockey} and~\ref{sec-ps} we explain how to
deduce the hockey stick conjecture (Theorem~\ref{thm-hockey}) and the
ball conjecture (Theorem~\ref{thm-ps}) from Theorem~\ref{thm-dd}. 

Altogether, we obtain a new self-contained proof of these three
results. 
Yet, our versions of the three conjectures are a bit less precise than
the corresponding results of Hoffman, Johnson and Junge.
In particular, we show an exponential bound only on the lower tail
of~$S_n$, while they also establish an
exponential bound on the upper tail, which is more involved.
As a consequence, we obtain less detailed convergence results, and we
do not get a control on the transition from polynomial to exponential
stabilization time of the model on a cycle, for which our
one-sided concentration bound does not seem to be enough: see
Proposition 8.12 in~\cite{HJJ1}.

That being said,
 we manage to obtain the three conjectures even without an exponential
bound on the upper tail, using tricks to bypass this.
In particular, in Section~\ref{sec-ps} about the ball conjecture, we
want to sum the probabilities of events
of the type~$\{S_n/n\geq\rs+\varepsilon\}$, and if the number of terms
is of order~$n$, one would need to know that the probabilities
decrease fast enough to beat this entropic factor~$n$.
But we are able to avoid this issue by taking advantage of the
margin~$\varepsilon$
in order to end up with a sum of only finitely many terms (this
finite number being large if~$\varepsilon$ is small, but not depending
on~$n$), which allows us to conclude even knowing only that the
probabilities tend to~$0$ but not at which speed.
For more details about this we refer the reader to
Section~\ref{sec-ps} and to the comment at the very end of
Section~\ref{sec-dd}.

\subsection{Self-organized criticality and sandpiles}

One main motivation to study Activated Random Walks is the
quest for a simple model which exhibits self-organized criticality, a
concept coined in by the physicists Bak, Tank and
Wiesenfeld~\cite{BTW87} to describe the behaviour of certain systems
which are spontaneously attracted to a critical-like state.
Unlike ordinary phase transitions, where the critical regime is only
observed for a very special choice of the parameters, self-organized
criticality means that a critical regime is reached without fine
tuning of the parameters.

In Section~\ref{sec-presentation}, we saw that the conservative
dynamics of Activated Random Walks on the infinite lattice (without particle addition or dissipation) exhibits a phase transition in the usual sense,
with two phases separated by a threshold density~$\rc$ (which, if
considered as a function of the sleep rate~$\lambda$, can also be seen
as a critical curve).

Self-organized criticality comes into play when one considers the
driven-dissipative chain described in Section~\ref{sec-statement-dd}.
Thanks to these two mechanisms of addition of particles and
dissipation at the boundaries when the segment becomes
too crowded to accommodate more particles, the system is able to
self-tune to the critical density.
Moreover, the hockey stick conjecture shows that the critical density
is not only reached as the limit after a very large number of steps of
the chain, but is achieved as soon as at least a critical density of
particles has been added to the system.

A stronger version of this conjecture, recently settled
in~\cite{HJJM25}, shows that the driven-dissipative Markov chain
exhibits cutoff exactly at the critical density.
This shows that not
only the density is rapidly self-tuned to the critical density, but
the distribution of the sleeping particles quickly resembles the
stationary distribution.

The idea to achieve self-organized criticality through the
introduction of driving and dissipation in a conservative model which
presents a usual phase transition is more general.
It is already present in the Abelian Sandpile Model, which was
suggested by Bak, Tang and Wiesenfeld to exemplify their concept.
In this model, there is only one type of particle, and a vertex of
the graph is declared unstable when the number of particles on it
exceeds the degree of the vertex.
Unstable sites may topple, sending one particle to each neighbouring
site, which may in turn make some of these neighbours become unstable.
As for Activated Random Walks, one may construct a
driven-dissipative Markov chain for the Abelian Sandpile and study its
stationary distribution.

Due to its more deterministic nature, the Abelian Sandpile Model is
more amenable to exact computations.
This enables the study of its
stationary distribution, which indeed presents critical-like
features such as power law correlations and fractal
structures~\cite{Dhar06,Redig06,Jarai18}.
However, the density conjecture and the hockey stick conjecture,
which were originally predicted for the Abelian Sandpile, turn out to
fail for this model: this has been proved rigorously on some
particular graphs and suggested with numerical simulations on the
two-dimensional lattice~\cite{FLW10a,FLW10b,JJ10}.

In view of these defects, Activated
Random Walks emerged as a variant of the Abelian Sandpile Model
involving more randomness, along
with another variant called the Stochastic
Sandpile Model, which was less studied but
 is expected to behave similarly.
Conjectures about the self-critical behaviour of these two models were
formulated progressively
in~\cite{DMVZ00,DRS10,Rolla20,LS24}.
Recent results about the mixing time of Activated Random Walks
suggest an explanation for the fact that this model behaves better
than the Abelian sandpile, because it mixes faster~\cite{LL21,BS22}.
For a broader comparative overview of sandpile models and Activated
Random Walks, we refer the reader to~\cite{HJJ2} and to the references
therein.

Yet, an important question which remains open for now is that of the
correlations in the stationary distribution of the driven-dissipative
system.
It is expected that the stationary distributions
converge to a limiting distribution on the infinite lattice, and
that the correlations in this limiting distribution decay as
power laws in the distance, indicating the absence of a
characteristic scale.
This would qualify the state reached by
the system as truly critical.

\section{Main ingredients and some perspectives}
\label{sec-ingredients}

Section~\ref{sec-superadd} below presents our main innovation, the
superadditivity property, along with the convergence property which
follows from it.

We then introduce the other ingredients on which we rely, so that the
rest of the paper is self-contained.
In particular, Section~\ref{sec-sampling} explains a technique
of~\cite{LL21} which allows to easily sample from the
stationary distribution, and thus to study~$S_n$.
Then, in Section~\ref{sec-fraction} we state two results
of~\cite{RT18} and~\cite{For24} relating the phase transition of the
model to the number of particles jumping out of a segment, and
in Section~\ref{sec-abelian} we present the site-wise representation
of the model and its Abelian property.
Several open perspectives are then outlined in Section~\ref{sec-open}.

\subsection{The crucial point: superadditivity}
\label{sec-superadd}

Recall the notation~$S_n$ introduced in Section~\ref{sec-statement-dd}
for the number of particles in the segment~$V_n=\{1,\,\ldots,\,n\}$
under the stationary
distribution of the driven-dissipative chain.
The main result of this paper is the following proposition, which
shows an almost superadditivity property for~$S_n$.

A similar superadditivity argument was already at the
heart of the proof of the density conjecture in~\cite{HJJ1}, but not
for~$S_n$.
Their approach uses another variable for which superadditivity is more
direct, but some work is then needed to relate this quantity to~$S_n$.
Here, we use a different approach, showing directly the superadditivity property for~$S_n$.

\begin{proposition}
\label{prop-superadd}
For every~$n,m\geq 1$ the variable~$S_{n+m+1}$ stochastically
dominates the sum of~$S_n$ and of a copy of~$S_m$ which is independent
of~$S_n$.
\end{proposition}

Note that this is only a property of the distribution of~$S_n$, but we
state it with random variables for concreteness.
The proof of this proposition, which is the main added value of this
paper, is presented in Section~\ref{sec-proof-superadd}.
We welcome any ideas to obtain a similar superadditivity property
without the~$+1$ term,
that is to say, to show that~$S_{n+m}$ dominates an independent sum
of~$S_n$ and~$S_m$.
Yet, this small defect is harmless for what we are interested in,
because if for every~$n\geq 1$ we consider~$X_n=S_{2n-1}$, then
Proposition~\ref{prop-superadd} entails that this
sequence~$(X_n)_{n\geq 1}$ is
stochastically superadditive, in the following sense:

\begin{definition}
\label{def-superadd}
We say that a sequence of real variables~$(X_n)_{n\geq 1}$ is stochastically
superadditive if for every integers~$n,\,m\geq 1$ the variable~$X_{n+m}$
stochastically dominates the sum of~$X_n$ and of a copy of~$X_m$ which
is independent of~$X_n$.
\end{definition}

This superadditivity property has the following consequence, whose
elementary proof is
presented in Section~\ref{sec-cv}.

\begin{lemma}
\label{lemma-cv}
Let~$(X_n)_{n\geq 1}$ be a sequence of non-negative random variables
which is stochastically superadditive.
Then, defining~$\rs=\sup_{n\geq 1}\E X_n/n\in[0,\,\infty]$, we have the convergence
in probability~$X_n/n\to\rs$ as~$n\to\infty$.
Moreover, we have the following exponential bound on the lower tail:
for every~$\rho<\rs$ there
exists~$c>0$ such that, for every~$n$ large enough,
\begin{equation}
\label{lower-bound}
\PP\bigg(\frac{X_n}n\leq \rho\bigg)
\ \leq\ 
e^{-cn}
\,.
\end{equation}
\end{lemma}

In the case of~$S_n$, this lemma enables to deduce
that~$S_n/n$ converges in probability as~$n\to\infty$, with an
exponential bound on the lower tail.

Note that no exponential bound on the upper tail follows from
stochastic superadditivity in general, as shown by the following
counter-example.
Consider the symmetric simple random walk on~$\Z^d$ with~$d\geq 3$ and
denote by~$R_n$ its range, i.e., the number of distinct sites visited
by the walk until step~$n$.
Then, the sequence~$(X_n)_{n\geq 1}$ defined by~$X_n=n-R_n$ is
stochastically superadditive (because~$R_n$ is stochastically
subadditive).
The walk being transient, we have~$\rs=\sup\E X_n/n<1$ and, for
every~$\rho\in(\rs,\,1)$, there exists~$c>0$ such
that, for~$n$ large enough,
\[
\PP\bigg(\frac{X_n}n\geq\rho\bigg)
\ =\ 
\PP\bigg(\frac{R_n}n\leq 1-\rho\bigg)
\ \geq\ 
\exp\big(-c\,n^{1-2/d}\big)
\,,
\]
this estimate being obtained by considering the probability that the
walk stays confined inside a ball of volume~$(1-\rho)n$
and using that during~$n^{2/d}$ steps the walk stays in a
ball of radius~$n^{1/d}$ with positive probability.

However, an exponential bound on the upper tail is not required to
establish Theorems~\ref{thm-dd},~\ref{thm-hockey} and~\ref{thm-ps}, for which
convergence in probability of~$S_n/n$ is enough.

\subsection{Exact sampling}
\label{sec-sampling}

We now state a nice key result of~\cite{LL21}
which gives a convenient way to sample~$S_n$:

\begin{lemma}
\label{lemma-sampling}
Fix an integer~$n\geq 1$.
Consider the initial configuration with one active particle on each
site of~$V_n$ and let the system evolve, with particles being killed
when they jump out of~$V_n$, until no active particle remains
in~$V_n$.
Then the distribution of the resulting stable configuration is exactly
the stationary distribution of the driven-dissipative Markov chain
on~$V_n$ defined in Section~\ref{sec-driven-dissipative}.
In particular, the number of sleeping particles remaining in~$V_n$ is
distributed as~$S_n$.
\end{lemma}

This property is not specific to the one-dimensional case and holds
more generally.
Its quite elementary proof goes as follows: let~$\eta$ be the random
configuration obtained after stabilizing the initial configuration
with one active particle per site, and let~$\eta'$ be the
configuration obtained after adding one active particle to~$\eta$ at a
site~$X\in V_n$ chosen uniformly at random, and stabilizing.
Then the Abelian property (see Section~\ref{sec-abelian}) shows
that~$\eta'$ can also be obtained by directly stabilizing the
configuration~$\1_{V_n}+\delta_X$, and this can be done by first
letting the extra particle at~$X$ walk until it jumps out of~$V_n$,
and then stabilizing the remaining~$n$ particles.
This shows that~$\eta'$ has the same distribution as~$\eta$, implying
that the distribution of~$\eta$ is the invariant distribution of
the driven-dissipative chain.
See~\cite{LL21} for more explanations about this proof.

\subsection{Fraction jumping out of a segment}
\label{sec-fraction}

Once the convergence in law of~$S_n/n$ is established, to identify the
limit as~$\rc$ (Theorem~\ref{thm-dd}) and to obtain the hockey stick
conjecture (Theorem~\ref{thm-hockey}), 
we rely on the two following results.
They give information on~$M_n$, which is the number of particles
jumping out of~$V_n$ during stabilization of~$V_n$ with particles
being killed when they exit~$V_n$ (but not necessarily starting with
one active particle per site).
One the one hand, to show the inequality~$\rs\geq\rc$, we use the
following consequence of a
result of~\cite{RT18}:

\begin{lemma}
\label{lemma-RT}
For each~$\rho<\rc$,
if the initial configuration is i.i.d.\ with
density of particles~$\rho$, then~\smash{$\E M_n=o(n)$}.
\end{lemma}

To show the other inequality~$\rs\leq\rc$, we use this result
of~\cite{For24} which is a kind of a reciprocal:

\begin{lemma}
\label{lemma-fraction-exits}
For every~$\rho>\rc$ there
exist~$\varepsilon>0$ and~$c>0$ such that for every~$n\geq 1$, for
every deterministic initial configuration~$\eta:V_n\to\N$ with at
least~$\rho n$ particles, all active, we have~$\PP_\eta(M_n>\varepsilon
n)\geq c$ (where~$\PP_\eta$ denotes the model started from~$\eta$).
\end{lemma}

Our proofs of Theorems~\ref{thm-dd} and~\ref{thm-hockey} using these
two results are presented in Sections~\ref{sec-dd} and~\ref{sec-hockey}.
Note that Theorem~\ref{thm-dd} (along with a monotonicity property
given by Lemma~\ref{lemma-mon-active} below) implies a more precise
version of Lemma~\ref{lemma-fraction-exits}, namely that for
every~$\varepsilon<\rho-\rc$ we
have~$\inf_{\|\eta\|\geq\rho n}
\PP_\eta(M_n>\varepsilon n)\to 1$ as~$n\to\infty$, where the infimum
is taken over all the initial configurations~$\eta:V_n\to\N$ with at
least~$\rho n$ particles, all active.

Then, in Section~\ref{sec-ps} we present the proof of
Theorem~\ref{thm-ps}, about the point source case.
For the outer bound, we rely on another result of~\cite{For24},
which
shows that if only few particles jump out of a segment, then with high
probability no particle leaves a slightly larger segment:

\begin{lemma}
\label{lemma-NML}
For every~$n\geq 1$, for every deterministic initial
configuration~$\eta:V_n\to\N$, for
any~$i,\,j\in\N$ we have
\[
\PP\big(A(\eta)\subset\{1-2j,\,\ldots,\,n+2j\}\big)
\ \geq\ 
\PP_\eta\big(M_n\leq i\big)
\times
\PP\big(G_1+\cdots+G_i\leq j\big)
\,,
\]
where~$A(\eta)$ denotes the set of sites which are visited
during the stabilization of~$\eta$ in~$\Z$ and~$G_1,\,\ldots,\,G_i$
are i.i.d.\ geometric variables with
parameter~$\lambda/(1+\lambda)$.
\end{lemma}

\subsection{Abelian property and monotonicity}
\label{sec-abelian}

The Abelian property is a central tool to study Activated Random
Walks.
It comes with a graphical representation of this kind of interacting
particles system, usually called the Diaconis and Fulton
representation~\cite{DF91,RS12}.

For every site~$x\in\Z$ consider an infinite
sequence~$(\tau_{x,j})_{j\geq 1}$
of ``instructions'', where each instruction~$\tau_{x,j}$ can either be
a ``sleep instruction'' or a ``jump instruction'' to some site~$y\in\Z$.
Let~$\tau=(\tau_{x,j})_{x\in\Z,j\geq 1}$ be such an array of
instructions, let~$\eta:\Z\to\N\cup\{\s\}$ be a particle
configuration, and let~$h:\Z\to\N$ be an array called odometer, which
counts how many instructions of~$\tau$ have already been used at
each site.
Recall that a site~$x\in\Z$ is called unstable in~$\eta$ if there is
at least one active particle at~$x$.
If~$x\in\Z$ is unstable, we say that it is legal to topple the
site~$x$.
Toppling the site~$x$ means applying the next
instruction from the array~$\tau$ at~$x$, namely~$\tau_{x,\,h(x)+1}$,
to the configuration~$\eta$.
If this instruction is a jump instruction to some site~$y\in\Z$, then
we make one particle jump from~$x$ to~$y$, waking up the sleeping
particle at~$y$ if any.
If this instruction is a sleep instruction, then the particle at~$x$
falls asleep if~$\eta(x)=1$, and nothing happens if~$\eta(x)\geq 2$.
This gives a new configuration~$\eta'$ and a new
odometer~$h'=h+\delta_x$.

If~$\alpha=(x_1,\,\ldots,\,x_k)$ is a finite sequence of sites
of~$\Z$, we say
that~$\alpha$ is a legal toppling sequence for~$\eta$ if it is legal to topple
these sites in this order.
The odometer of a toppling sequence~$\alpha$, which counts how many
times each site appears in~$\alpha$, is defined
as~$m_\alpha=\delta_{x_1}+\cdots+\delta_{x_k}$.

For a fixed configuration~$\eta$ and a fixed array of
instructions~$\tau$, for every~$V\subset\Z$ we can define the
stabilization odometer of~$V$, which is given by
\[
m_{V,\,\eta}
\ =\ 
\sup_{\alpha\subset V,\ \alpha\text{ legal for }\eta}\,m_\alpha
\,,
\]
where the notation~$\alpha\subset V$ simply means that all the sites
of~$\alpha$ belong to~$V$.

Another useful notion is that of acceptable topplings.
We say that it is acceptable to topple a site~$x$ if~$\eta$ contains at
least one particle at~$x$, which may be sleeping.
When we perform an acceptable toppling at a site which contains a
sleeping particle, we first wake it up.
We say that a sequence of topplings~$\alpha=(x_1,\,\ldots,\,x_k)$ is
acceptable if it is acceptable to topple these sites in this order.

We say that an acceptable sequence of topplings stabilizes~$\eta$
in~$V$ if, in the final configuration obtained after performing these
topplings, all the sites of~$V$ are stable.

For every fixed~$\eta$ and~$\tau$, we have the following least action
principle, which compares acceptable and legal sequences of topplings
and may also be called monotonicity with respect to
enforced activation:

\begin{lemma}[Lemma~2.1 in~\cite{Rolla20}]
\label{lemma-monotonicity}
For every~$V\subset\Z$, if~$\alpha$ is an acceptable sequence of
topplings that
stabilizes~$\eta$ in~$V$, and~$\beta\subset V$ is a legal sequence of
topplings for~$\eta$, then~$m_\alpha\geq m_\beta$.
Thus, if~$\alpha$ is an acceptable sequence of topplings that
stabilizes~$\eta$ in~$V$, then~$m_\alpha\geq m_{V,\,\eta}$.
\end{lemma}

This entails in particular that if~$\alpha$ and~$\beta$ are two legal
sequences of topplings in~$V$ which stabilize~$\eta$ in~$V$,
then~$m_\alpha=m_\beta=m_{V,\,\eta}$ (hence the name stabilization
odometer), and it is not hard to convince
oneself that the resulting final configurations are also equal.
This last equation is known as the abelian property,
which allows us to choose
whatever order to perform the topplings, as soon as they are legal, or
acceptable if we only look for upper bounds on the stabilization
odometer.

For a given finite set~$V\subset\Z$, a configuration~$\eta:V\to\N\cup\{\s\}$
and an array~$\tau=(\tau_{x,j})_{x\in V,\,j\geq 1}$ we
define~$\Stab_V(\eta,\,\tau)$ as the number of (sleeping) particles
which remain in~$V$ after applying any legal sequence of topplings
in~$V$ which
stabilizes~$\eta$ in~$V$ (so that this corresponds to ignoring
particles once they jump out of~$V$).

To relate this construction to the dynamics of Activated Random Walks,
one needs to make the array of instructions random.
Recalling that we fixed throughout this article a sleep
rate~$\lambda>0$ and~$p\in(0,1)$ a probability to jump to the left,
we consider i.i.d.\ stacks of instructions, where each instruction is
a sleep instruction with
probability~$\lambda/(1+\lambda)$, a jump instruction to the left
neighbouring site with
probability~\smash{$p/(1+\lambda)$}, and a jump instruction to the right with
probability~$(1-p)/(1+\lambda)$.
Throughout the article probabilities are all denoted by~$\PP$,
which refers most of the time to this distribution on the arrays,
but that we also abusively use when there are other random elements
than just the array~$\tau$ (for example the initial configuration).

Let us conclude this section with the following monotonicity property, which is a
consequence of
Lemma~\ref{lemma-monotonicity}:

\begin{lemma}
\label{lemma-mon-active}
Let~$V$ be a finite subset of~$\Z$ and let~$\eta,\,\xi:V\to\N$ be two
deterministic
configurations of particles on~$V$ which contain only active
particles and which are such that~$\eta\geq\xi$.
Let~$\tau$ be a random array of instructions with
distribution~$\PP$.
Then~$\Stab_V(\eta,\,\tau)$ stochastically
dominates~$\Stab_V(\xi,\,\tau)$.
\end{lemma}

\begin{proof}
It is enough to consider the case~$\eta=\xi+\delta_x$ for a
certain~$x\in V$.
Then, this configuration~$\xi+\delta_x$ in~$V$ can be stabilized by
first forcing one particle to walk from~$x$ until it jumps out of~$V$,
using acceptable topplings, and then stabilizing the remaining
configuration~$\xi$ in~$V$.
This yields an acceptable sequence of topplings which
stabilizes~$\xi+\delta_x$ in~$V$ and which leaves a number of sleeping
particles in~$V$ which is equal to~$\Stab_V(\xi,\,\tau)$ in
distribution.
By virtue of Lemma~\ref{lemma-monotonicity}, this shows the desired
stochastic domination.
\end{proof}

\subsection{Some open questions about possible generalizations}
\label{sec-open}

Let us recall that our results
are limited to the one-dimensional
case, with nearest-neighbour jumps (as is the case of~\cite{HJJ1}
and~\cite{HJJ2} which only deal with symmetric jumps in one
dimension).
Not only is our proof of superadditivity very specific to dimension~1,
but also our proof of the upper bound~$\rs\leq\rc$ in
Section~\ref{sec-rs-leq-rc}, which uses
Lemma~\ref{lemma-fraction-exits} which is only known in dimension~1,
and also our proof of the outer bound in the ball conjecture in
Section~\ref{sec-outer}, which uses Lemma~\ref{lemma-NML} which is
also only proved in dimension~1.

Thus, a natural open question is the following: which of these results
also hold in higher dimension, on more general graphs or with more
general jump distributions?
In particular, in view of the crucial role played by the superadditivity property, it
would be particularly interesting to know whether a similar
superadditivity property still holds.
And, if this is the case, which results could one deduce from such a
superadditivity property?

One intuitive strategy to extend to other graphs would be to determine
whether it remains true that 
turning one site into an ejector seat 
stochastically decreases the number of sleeping particles remaining
after stabilization.
Or, more generally, is it true that on an oriented graph, diverting
one oriented edge to the sink (i.e., the edge leads directly out of
the graph) stochastically decreases the number of sleeping particles
remaining?
This would in particular imply the superadditivity property without
the~$+1$ defect.

Besides, let us cite another natural question: is there a simple way to
establish the exponential bound on the upper tail of~$S_n/n$?

More generally, let us highlight that many of the open problems presented
in~\cite{LS24} remain open, including the density conjecture in
dimension~$d\geq 2$.

\section{Superadditivity: proof of Proposition~\ref{prop-superadd}}
\label{sec-proof-superadd}

\subsection{The setting: two segments on both sides of the origin}

Let~$n,\,m\geq 1$, consider the segment~$V=\{-n,\,\ldots,\,+m\}$ and let~$\tau=(\tau_{x,j})_{x\in
V,\,j\geq 1}$ be a random array of instructions with the distribution
described in Section~\ref{sec-abelian}.
Let us write~$L=\{-n,\,\ldots,\,-1\}$
and~\smash{$R=\{1,\,\ldots,\,\}$}, so that~$V=L\cup\{0\}\cup R$.
Recalling the notation~$\Stab$ introduced in Section~\ref{sec-abelian},
which counts the number of particles left after stablizing a given
configuration, let us define
\begin{align*}
S_V
&\ =\ 
\Stab_V(\1_V,\,\tau)
\ \stackrel{d}{=}\ S_{n+m+1}
\,,
\\
S_L
&\ =\ 
\Stab_L(\1_L,\,\tau)
\ \stackrel{d}{=}\ S_{n}
\,,
\\
S_R
&\ =\ 
\Stab_R(\1_R,\,\tau)
\ \stackrel{d}{=}\ S_{m}
\,,
\end{align*}
where the equalities in distribution follow from the exact sampling
result given by Lemma~\ref{lemma-sampling}.
Notice that~$S_L$ and~$S_R$ are independent because they depend on
instructions of~$\tau$ on two disjoint sets of sites.
Thus, our goal is to show that~$S_V$ stochastically
dominates~$S_L+S_R$.
Note that there is no hope to show more than a
stochastic domination there, because it is not true
that~$S_V\geq S_L+S_R$ for all the possible realizations of the array of
instructions.

At this point, notice that~$S_L+S_R$ corresponds
to the number of particles which remain sleeping in the segment~$V$ if
in some sense the site~$0$ becomes a sink or, formulated differently,
if all instructions at this site~$0$ are changed into jump
instructions which make particles jump directly out of~$V$.
Then, the intuitive idea is that changing these instructions to make
them point out of~$V$ decreases, at least in distribution, the number
of particles which remain in~$V$.
Although intuitive, this assertion is less obvious than it seems, and
its proof below uses the trick to stabilize the segment from one side
to the other side, which is very specific to the
one-dimensional case.
Ideas about what happens in higher dimension are most welcome!

\subsection{Turning the origin into an ejector seat}

We now wish to study what happens when instructions at~$0$ are replaced, one
by one, with instructions jumping out of~$V$.
One important detail is that, instead of starting with the first
instruction, we choose to replace the instructions one by
one ``starting from infinity''.
This will turn out to be crucial in the proof.
Hence, for every~$k\geq 1$, we let~$\tau_k$ denote the array obtained from~$\tau$
by replacing all the instructions with number~$j\geq
k$ at the site~$0$ with a jump instruction pointing out of~$V$,
i.e., for example a jump instruction from~$0$ to~$m+1$.
We then consider the number of particles left sleeping in~$V$ when
stabilizing~$\1_V$ using these
modified arrays of instructions,
defining~$N_k=\Stab_V(\1_V,\,\tau_k)$ for every~$k\geq 1$.

On the one hand, since in~$\tau_1$ all the instructions at the
site~$0$ are jumps
directly to the exterior of~$V$, when using instructions
in~$\tau_1$ any particle which visits the site~$0$ will jump out
of~$V$ without visiting~$L\cup R$.
Therefore, we have~$N_1=S_L+S_R$.

On the other hand, if~$k$ is strictly larger than the number of
instructions used at~$0$ to stabilize~$\1_V$ in~$V$ with instructions read
from~$\tau$, then~$N_k=S_V$ because the modified instructions are not
used anyway.
Since this number of instructions used at~$0$ is almost surely
finite, we deduce that~$N_k$ almost surely converges to~$S_V$
when~$k\to\infty$.

Thus, the proof of the proposition will be complete if we show that
for every~$k\geq 1$ the variable~$N_{k+1}$ stochastically
dominates~$N_k$.
That is to say, we study the replacement of the~$k$-th instruction
at the site~$0$, in
a setting where all the instructions coming after this one are already
replaced with jumps to the exterior.

\subsection{Replacing one instruction with a jump to the
exterior}

Let~$k\geq 1$.
First, assume that~$\tau_{0,k}$ is a sleep
instruction.
If moreover this instruction is the last instruction
used at~$0$ when stabilizing~$\1_V$ using the array~$\tau_{k+1}$, then
we have~$N_{k+1}=N_k+1$ because the two resulting final configurations
only differ by the presence of a sleeping particle at~$0$.
Otherwise, if the stabilization of~$\1_V$ using~$\tau_{k+1}$ uses
strictly less or strictly more than~$k$ instructions at~$0$,
then~$N_{k+1}=N_k$.
Therefore, on the event~$\cA$ that~$\tau_{0,k}$ is a sleep instruction
we have~$N_{k+1}\geq N_k$.

Let now~$\cB$ be the event that~$\tau_{0,k}$ is a jump
instruction to the right, so that at this position in~$\tau_{k+1}$
there is a jump to the right while in~$\tau_k$ there is a direct jump out
of~$V$.
Let us consider the following procedure
to stabilize
the configuration~$\1_V$ in~$V$ using the instructions
in~$\tau_k$:

\begin{enumerate}
\item
During the first step, we always topple the leftmost unstable site
in~$V$, and we stop as soon as~$k$ instructions have been used at~$0$ or
there are no more active particles in~$V$ (if this happens before
using~$k$ instructions at~$0$).

\item
During the second step, we stabilize the remaining active particles
in~$V$ using whatever legal sequence of topplings in~$V$.
\end{enumerate}

Let~$\eta_1$ and~$h_1$ be the random configuration on~$V$ and the
random odometer obtained at the end of step~1, and
let~$\eta:V\to\N\cup\{\s\}$ and~$h:V\to\N$ be a fixed configuration
and a fixed odometer such that~$\PP(\cB_{\eta,h})>0$,
where~$\cB_{\eta,h}=\cB\cap\{\eta_1=\eta,\,h_1=h\}$.
We wish to show that, conditionally on this event~$\cB_{\eta,h}$, the
variable~$N_{k+1}$ stochastically dominates~$N_k$.

First, if~$h(0)<k$ then it means that, on the event~$\cB_{\eta,h}$,
step~1 stopped before using~$k$
instructions at~$0$, so the configuration~$\eta$ is stable in~$V$
and~$N_{k+1}=N_k$.

Assume now that~$h(0)=k$ (note that by definition of step~1 it is
impossible
that~$h(0)>k$).
This means that (on the event~$\cB_{\eta,h}$) the last instruction
of~$\tau_k$ used during step~1
was the~$k$-th instruction at~$0$, which is a jump to the exterior.
Because of the priority to the left, this entails that the
configuration~$\eta$ is stable in~$L$.
Besides, during step~2 all the remaining instructions at~$0$ are jumps
directly to the exterior, so that any particle which visits the
site~$0$ during
step~2 then jumps out of~$V$ without visiting~$L\cup R$.
Therefore, on the event~$\cB_{\eta,h}$ we have
\begin{equation}
\label{eqNk}
N_k
\ =\ 
\|\eta\|_L
+\Stab_R(\eta,\,\tau')
\,,
\end{equation}
where~$(\tau')_{x\in R,\,j\geq 1}$ is the array on the sites of~$R$
obtained from~$\tau$ by removing the
instructions below the odometer~$h$ and re-indexing the remaining
instructions, that is to say, for every~$x\in R$ and~\smash{$j\geq
1$} we have~\smash{$\tau'_{x,j}=\tau_{x,\,h(x)+j}$}.

Now, recall that~$N_{k+1}$ is the number of particles that we obtain
if we use instead the array~$\tau_{k+1}$.
On the event~$\cB$ this array has a jump to the right at position~$k$
at site~$0$, so that applying the same sequence of topplings that we
performed during step~1 but reading the instructions
from~$\tau_{k+1}$ we obtain the configuration~$\eta_1+\delta_1$ (one
particle landed on site~$1$ instead of jumping to the exterior).
Hence, on the event~$\cB_{\eta,h}$ it holds
\begin{equation}
\label{eqNk1}
N_{k+1}
\ =\ 
\|\eta\|_L
+\Stab_R(\eta+\delta_1,\,\tau')
\,,
\end{equation}
with the same array~$\tau'$ of remaining instructions.

We now wish to apply Lemma~\ref{lemma-mon-active} to show that,
conditionally on~$\cB_{\eta,h}$, the
variable~$\Stab_R(\eta+\delta_1,\,\tau')$ stochastically
dominates~$\Stab_R(\eta,\,\tau')$.

First, recall that this Lemma only applies with configurations where all particles
are active.
Yet, note that the priority rule to the left
ensures that during step~1 there can never be a sleeping
particle to the right of an active particle.
As a consequence, the configuration~$\eta$ does not contain any
sleeping particle in~$R$ (because the last instruction was used
at~$0$).
See Figure~\ref{fig-superadd} for an example.

\begin{figure}
\begin{center}
\begin{tikzpicture}[scale=0.66]
\draw (-4,1) node[scale=1.5]{$\rightarrow$}
(-4,2) node[scale=1.5]{$\s$}
(-4,3) node[scale=1.5]{$\leftarrow$}
(-4,4) node[scale=1.5]{$\s$}
(-3,1) node[scale=1.5]{$\leftarrow$}
(-3,2) node[scale=1.5]{$\s$}
(-3,3) node[scale=1.5]{$\rightarrow$}
(-3,4) node[scale=1.5]{$\leftarrow$}
(-2,1) node[scale=1.5]{$\s$}
(-2,2) node[scale=1.5]{$\leftarrow$}
(-2,3) node[scale=1.5]{$\rightarrow$}
(-2,4) node[scale=1.5]{$\s$}
(-1,1) node[scale=1.5]{$\leftarrow$}
(-1,2) node[scale=1.5]{$\rightarrow$}
(-1,3) node[scale=1.5]{$\rightarrow$}
(-1,4) node[scale=1.5]{$\rightarrow$}
(0,1) node[scale=1.5]{$\rightarrow$}
(0,2) node[scale=1.5]{$\rightarrow$}
(0,3) node[scale=1.5][red!80!black]{$\Uparrow$}
(0,4) node[scale=1.5][red!80!black]{$\Uparrow$}
(1,1) node[scale=1.5]{$\s$}
(1,2) node[scale=1.5]{$\leftarrow$}
(1,3) node[scale=1.5]{$\rightarrow$}
(1,4) node[scale=1.5]{$\leftarrow$}
(2,1) node[scale=1.5]{$\rightarrow$}
(2,2) node[scale=1.5]{$\leftarrow$}
(2,3) node[scale=1.5]{$\s$}
(2,4) node[scale=1.5]{$\leftarrow$}
(3,1) node[scale=1.5]{$\leftarrow$}
(3,2) node[scale=1.5]{$\rightarrow$}
(3,3) node[scale=1.5]{$\rightarrow$}
(3,4) node[scale=1.5]{$\s$}
(4,1) node[scale=1.5]{$\s$}
(4,2) node[scale=1.5]{$\s$}
(4,3) node[scale=1.5]{$\rightarrow$}
(4,4) node[scale=1.5]{$\leftarrow$};

\foreach \x in {(-4,1), (-4,2), (-3,1), (-3,2), (-2,1), (-2,2),
(-2,3), (-1,1), (-1,1), (-1,2), (0,1), (0,2), (1,1), (1,2), (2,1),
(2,2), (3,1)}
{
\draw \x node[blue]{\textbf{/}};
}

\draw[blue,dashed, ultra thick, cap=round] (-4.5,2.5) node[left]{$h_1$} -- (-2.5,2.5) -- (-2.5,3.5)
-- (-1.5,3.5) -- (-1.5,2.5) -- (2.5,2.5) -- (2.5,1.5) -- (3.5,1.5) --
(3.5,0.5) -- (4.5,0.5);

\draw[orange, very thick] (-.5,1.5) -- (-.5,2.5) -- (.5,2.5) -- (.5,1.5)
-- cycle;

\draw[green!50!black, very thick]
(-4,0) node[scale=1.5]{$\s$}
(-3,0) node[scale=1.5]{$\s$}
(-2,0) node[scale=1.5]{$0$}
(-1,0) node[scale=1.5]{$0$}
(0,0) node[scale=1.5]{$0$}
(1,0) node[scale=1.5]{$2$}
(2,0) node[scale=1.5]{$0$}
(3,0) node[scale=1.5]{$1$}
(4,0) node[scale=1.5]{$0$}
(-4.3,0) node[left]{$\eta_1+\delta_1:$}
(1,0) circle(.4);

\draw[very thick] (-4.5,-.5) -- node[midway, below]{$L$ (all sleeping)} (-.5,-.5)
(.5,-.5) -- node[midway, below]{$R$ (all active)} (4.5,-.5);
\draw[very thick,white] (-0.3,-.5) -- node[midway, below,black]{$0$}
(0.3,-.5);

\draw(1.5,5.5) node[purple]{instructions
to make the extra particle};
\draw[purple, dashed, ultra thick, cap=round]
(-.5,2.6) -- (2.5,2.6) -- (2.5,4.5) --node[midway,above]{
exit with acceptable topplings}  (.5,4.5)
--(.5,3.5) --
(-.5,3.5) -- cycle;

\draw (0,-2) node{Outcome of step 1 using the array~$\tau_3$};

\draw[thick,->,green!50!black] (4.6,1.6) node[right]{extra particle} --
(1.4,.3);

\begin{scope}[xshift=380]
\draw[orange,->,thick] (-.6,5) node[above]{instruction~$\tau_{0,2}$
replaced with an ejector seat} -- (-.3,2.6);

\draw (-4,1) node[scale=1.5]{$\rightarrow$}
(-4,2) node[scale=1.5]{$\s$}
(-4,3) node[scale=1.5]{$\leftarrow$}
(-4,4) node[scale=1.5]{$\s$}
(-3,1) node[scale=1.5]{$\leftarrow$}
(-3,2) node[scale=1.5]{$\s$}
(-3,3) node[scale=1.5]{$\rightarrow$}
(-3,4) node[scale=1.5]{$\leftarrow$}
(-2,1) node[scale=1.5]{$\s$}
(-2,2) node[scale=1.5]{$\leftarrow$}
(-2,3) node[scale=1.5]{$\rightarrow$}
(-2,4) node[scale=1.5]{$\s$}
(-1,1) node[scale=1.5]{$\leftarrow$}
(-1,2) node[scale=1.5]{$\rightarrow$}
(-1,3) node[scale=1.5]{$\rightarrow$}
(-1,4) node[scale=1.5]{$\rightarrow$}
(0,1) node[scale=1.5]{$\rightarrow$}
(0,2) node[scale=1.5][red!80!black]{$\Uparrow$}
(0,3) node[scale=1.5][red!80!black]{$\Uparrow$}
(0,4) node[scale=1.5][red!80!black]{$\Uparrow$}
(1,1) node[scale=1.5]{$\s$}
(1,2) node[scale=1.5]{$\leftarrow$}
(1,3) node[scale=1.5]{$\rightarrow$}
(1,4) node[scale=1.5]{$\leftarrow$}
(2,1) node[scale=1.5]{$\rightarrow$}
(2,2) node[scale=1.5]{$\leftarrow$}
(2,3) node[scale=1.5]{$\s$}
(2,4) node[scale=1.5]{$\leftarrow$}
(3,1) node[scale=1.5]{$\leftarrow$}
(3,2) node[scale=1.5]{$\rightarrow$}
(3,3) node[scale=1.5]{$\rightarrow$}
(3,4) node[scale=1.5]{$\s$}
(4,1) node[scale=1.5]{$\s$}
(4,2) node[scale=1.5]{$\s$}
(4,3) node[scale=1.5]{$\rightarrow$}
(4,4) node[scale=1.5]{$\leftarrow$};

\foreach \x in {(-4,1), (-4,2), (-3,1), (-3,2), (-2,1), (-2,2),
(-2,3), (-1,1), (-1,1), (-1,2), (0,1), (0,2), (1,1), (1,2), (2,1),
(2,2), (3,1)}
{
\draw \x node[blue]{\textbf{/}};
}

\draw[blue,dashed, ultra thick, cap=round] (-4.5,2.5) node[left]{$h_1$} -- (-2.5,2.5) -- (-2.5,3.5)
-- (-1.5,3.5) -- (-1.5,2.5) -- (2.5,2.5) -- (2.5,1.5) -- (3.5,1.5) --
(3.5,0.5) -- (4.5,0.5);

\draw[orange, very thick] (-.5,1.5) -- (-.5,2.5) -- (.5,2.5) -- (.5,1.5)
-- cycle;

\draw[green!50!black, very thick]
(-4,0) node[scale=1.5]{$\s$}
(-3,0) node[scale=1.5]{$\s$}
(-2,0) node[scale=1.5]{$0$}
(-1,0) node[scale=1.5]{$0$}
(0,0) node[scale=1.5]{$0$}
(1,0) node[scale=1.5]{$1$}
(2,0) node[scale=1.5]{$0$}
(3,0) node[scale=1.5]{$1$}
(4,0) node[scale=1.5]{$0$}
(-5,0) node{$\eta_1:$}
(1,0) circle(.4);

\draw[very thick] (-4.5,-.5) -- node[midway, below]{$L$ (all sleeping)} (-.5,-.5)
(.5,-.5) -- node[midway, below]{$R$ (all active)} (4.5,-.5);
\draw[very thick,white] (-0.3,-.5) -- node[midway, below,black]{$0$}
(0.3,-.5);

\draw (0,-2) node{Outcome of step 1 using the array~$\tau_2$};

\end{scope}

\end{tikzpicture}
\end{center}
\caption{
\label{fig-superadd}
In the proof of Proposition~\ref{prop-superadd}, possible situation of
the array of instructions and the configuration after step~1, when we
study the replacement of the second instruction at~$0$ with a jump
instruction to the exterior (``ejector seat'').
In the left case, there is an extra particle at site~$1$, which we can
force to exit with acceptable topplings, without waking up any
sleeping particle.
Double red arrows ($\Uparrow$) represent ejector seat instructions.
}
\end{figure}

Besides, the event~$\cB_{\eta,h}$ being measurable with respect to
the instructions~$(\tau_{x,j})_{x\in V,\,j\leq h(x)}$, it is
independent of the array~$\tau'$.
Therefore, conditionally on~$\cB_{\eta,h}$, the array~$\tau'$ has the
same distribution as~$\tau$.

Therefore, Lemma~\ref{lemma-mon-active} applies and allows us to
deduce that~$\Stab_R(\eta+\delta_1,\,\tau')$ dominates~\smash{$\Stab_R(\eta,\,\tau')$}.
Plugging this into~\eqref{eqNk} and~\eqref{eqNk1} it follows that,
conditionally on~$\cB_{\eta,h}$, the
variable~$N_{k+1}$ stochastically dominates~$N_k$.

This is true for every couple~$(\eta,\,h)$ such that~$\PP(\cB_{\eta,h})>0$,
and the
same argument works on the event~$\cC$ that
the instruction~$\tau_{0,k}$ is a jump to
the left (by considering instead the procedure which always topples the
rightmost active site during step~1), so we eventually conclude
that~$N_{k+1}$ stochastically dominates~$N_k$.

\section{Convergence for superadditive sequences: proof of
Lemma~\ref{lemma-cv}}
\label{sec-cv}

Let~$(X_n)_{n\geq 1}$ be a stochastically superadditive sequence of
non-negative variables.
First note that Fekete's subadditive lemma~\cite{Fekete23} entails
that~$\E X_n/n\to\rs$ as~$n\to\infty$, where~$\rs=\sup_{n\geq
1}\E X_n/n$.

\subsection{Exponential bound on the lower tail}

Let us start by directly proving the exponential
bound~\eqref{lower-bound} on the lower
tail of~$X_n/n$.
Let~$\rho<\rs$.
Let us fix~$N\geq 1$ such that~$\E X_N/N>\rho$.
Considering the normalized moment-generating function
\[
\Lambda\,:\,\theta\in\R_+
\ \longmapsto\ 
-
\frac 1 N
\ln\E\big[e^{-\theta X_N}\big]
\,,
\]
we have~$\Lambda(0)=0$ and~$\Lambda'(0)=\E X_N/N>\rho$.
Hence, we can find~$\theta>0$ such that~$\Lambda(\theta)>\rho\theta$.
Then, for every~$n\geq N$, using the superadditivity assumption we can
write
\[
\ln\E\big[e^{-\theta X_n}\big]
\ \leq\ 
\Big\lfloor\frac n N\Big\rfloor
\ln\E\big[e^{-\theta X_N}\big]
\ =\ 
\Big\lfloor\frac n N\Big\rfloor
\big(-N\Lambda(\theta)\big)
\ \stackrel{n\to\infty}{\sim}\ 
- n \Lambda(\theta)
\,.
\]
Therefore, if we choose~$c>0$ such
that~$\rho\theta+c<\Lambda(\theta)$, then
for~$n$ large enough we have
\[
\ln\E\big[e^{-\theta X_n}\big]
\ \leq\ 
-n(\rho\theta+c)
\,,
\]
which implies through a Chernoff inequality that
\[
\PP\bigg(\frac {X_n} n \leq \rho\bigg)
\ \leq\ 
\E\big[e^{-\theta X_n}\big]e^{\theta\rho n}
\ \leq\ 
e^{-cn}
\,,
\]
concluding the proof of~\eqref{lower-bound}.
Note that the above proof is more or less equivalent to using the
Gärtner-Ellis theorem (see for
example Theorem~2.3.6 in~\cite{DZ10}).

\subsection{Bound on the upper tail}

We now deduce the bound on the upper tail from the bound on the lower
tail.

Assume that~$\rs<\infty$, otherwise there is nothing to prove.
Fix~$\varepsilon>0$ and let us show that when~$n\to\infty$ we have~\smash{$\PP(X_n/n \geq\rs+\varepsilon)\to 0$}.

Let~$\delta>0$.
Take~$\varepsilon'=\varepsilon\delta/2$, and let~$c>0$ and~$n_0\geq 1$
be such that for~$\rho=\rs-\varepsilon'$ the bound~\eqref{lower-bound}
holds for every~$n\geq n_0$.
Then, for~$n\geq n_0$ we can write
\begin{align*}
\rs
\ \geq\ 
\E\bigg[\frac{X_n}n\bigg]
&\ \geq\ 
(\rs-\varepsilon')\,
\PP\bigg(\frac{X_n}n\geq \rs-\varepsilon'\bigg)
+(\varepsilon+\varepsilon')\,
\PP\bigg(\frac{X_n}n\geq \rs+\varepsilon\bigg)
\\
&\ \geq\ 
(\rs-\varepsilon')(1-e^{-cn})
+\varepsilon\,
\PP\bigg(\frac{X_n}n\geq \rs+\varepsilon\bigg)
\\
&\ \geq\ 
\rs-\varepsilon'-\rs e^{-cn}
+\varepsilon\,
\PP\bigg(\frac{X_n}n\geq \rs+\varepsilon\bigg)
\,,
\end{align*}
which implies that
\[
\PP\bigg(\frac{X_n}n\geq \rs+\varepsilon\bigg)
\ \leq\ 
\frac{\varepsilon'+\rs \,e^{-cn}}{\varepsilon}
\ =\ 
\frac\delta 2 +\frac{\rs}{\varepsilon}\,e^{-cn}
\ \leq\ 
\delta
\,,
\]
provided that~$n$ is large enough.
This concludes the proof of the convergence of~$X_n/n$ to~$\rs$ in
probability, thus concluding the proof of Lemma~\ref{lemma-cv}.

\section{Limit of the driven-dissipative chain: proof of
Theorem~\ref{thm-dd}}
\label{sec-dd}

We now show that~$S_n/n\to\rc$ in probability.
The superadditivy property given by
Proposition~\ref{prop-superadd}, along with its consequence given by Lemma~\ref{lemma-cv},
already ensures that~$S_n/n\to\rs$ in
probability, where~$\rs=\sup_{n\geq 1}\E S_n/n$.
Thus, to obtain Theorem~\ref{thm-dd} there only remains to show
that~$\rs=\rc$.

We treat separately the lower bound~$\rs\geq\rc$  and the upper
bound~$\rs\leq\rc$.
In both cases we rely on Lemma~\ref{lemma-sampling} which tells us
that~$S_n$ is the number of sleeping particles which remain in~$V_n$
after stabilization starting with one active particle per site.

\subsection{Lower bound}
\label{sec-dd-lower}

The lower bound is a quite direct consequence of
Lemma~\ref{lemma-RT}.
Let~$\rho<\rc$ and consider the following
procedure to stabilize the segment~$V_n$
starting with one active particle per site.

For each particle in the initial configuration, we draw an independent
Bernoulli variable with probability~$\rho$.
If we obtain~$0$, we force the particle to walk with acceptable
topplings until it jumps out of~$V_n$, and if we obtain~$1$ we leave
the particle at its starting point.
After this first step of the procedure, we obtain a configuration
which is i.i.d.\ with mean density~$\rho$, with all particles active.
Let~$N_n$ be the random number of particles in this configuration (it
is a binomial with parameters~$n$ and~$\rho$).

In the second step we simply perform legal topplings in whatever order
until a stable configuration is reached.
Let~$M_n$ be the number of particles which jump out of~$V_n$ during
this second step, so that the final number of sleeping particles
remaining in~$V_n$ after the two steps is equal to~$N_n-M_n$.
Since we performed acceptable topplings, the least action principle
given by Lemma~\ref{lemma-monotonicity} tells us that the odometer of the sequence that we
performed is above the legal stabilizing odometer.
The number of
particles jumping out of~$V_n$ being a non-decreasing function of the
odometer, we therefore have~$N_n-M_n\leq S_n$.
Besides, Lemma~\ref{lemma-RT} tells us that~$\E
M_n=o(n)$ as~$n\to\infty$, because~$\rho<\rc$.
Hence, we have
\[
\frac{\E S_n}{n}
\ \geq\ 
\frac{\E[N_n-M_n]}{n}
\ =\ 
\rho-\frac{\E M_n}{n}
\ \stackrel{n\to\infty}{\longrightarrow}\ 
\rho
\,,
\]
whence~$\rs\geq\rho$.
This being true for every~$\rho<\rc$, we get~$\rs\geq\rc$.

\subsection{Upper bound}
\label{sec-rs-leq-rc}

We now turn to the upper bound~$\rs\leq\rc$.
Let us assume by contradiction that~$\rs>\rc$.
Let~$\rho'\in(\rc,\,\rs)$, and let~$\varepsilon>0$ and~$c>0$ be given by
Lemma~\ref{lemma-fraction-exits}
applied with~$\rho'$, so that for every~$n\geq 1$, for every
deterministic initial configuration~$\eta:V_n\to\N$ with at
least~$\rho'
n$ particles, all active, 
with probability at least~$c$, at least~$\varepsilon n$ particles
exit.
Decreasing~$\varepsilon$ if necessary, we assume
that~$\rs-\varepsilon/2\geq\rho'$.
To shorten notation, let us define~$\delta=\varepsilon/2$.

Since~$S_n/n\to\rs$ in probability, we can take~$n$ large enough so
that~$\PP(S_n/n \leq\rs+\delta)\geq 1/2$.
Our aim is to deduce from this a lower bound on the probability
that~$S_n/n\leq\rs-\delta$, in order to contradict the fact that this
probability decays exponentially fast with~$n$.

To this end, a natural strategy is to start from one active
particle on each site and, during a first step, to perform legal topplings until less
than~$(\rs+\delta)n$ particles remain.
With probability at least~$1/2$ the number of particles indeed goes
below~$(\rs+\delta)n$ before stabilization.
We then would like to use
that, after this, if enough particles remain then at least~$2\delta n$
particles jump out with
probability at least~$c$.
But, to make this argument work, one must find a way to ensure that,
after the first step, when less than~$(\rs+\delta)n$ particles remain,
these particles are all active, otherwise the lower bound on the
probability that at least~$2\delta n$ more particles exit would not
apply.

Taking once again advantage of the one-dimensional setting, our idea
is that during the first stage we always topple the leftmost active
particle in the segment, until
enough particles jumped out from the left endpoint, and then repeat
this with the right endpoint (toppling the rightmost active particle).
Hence, to know how many particles we should expect to see jumping from each
endpoint, we need to translate the bound~$S_n/n\leq\rs+\delta$ into an
information on how many particles jump out on each side.

Let~$L_n$ and~$R_n$ denote the number
of particles which respectively
jump out of~$V_n$ from the left and from the right endpoint during
the stabilization of~$V_n$ starting with one active particle per site,
so that we have~\smash{$S_n=n-L_n-R_n$}.
Then, we can write
\[
\bigg\{
\frac{S_n}n\leq\rs+\delta\bigg\}
\ \subset\ 
\bigcup_{
\substack{0\,\leq\,\ell,\,r\,\leq\, n\\
n-\ell-r\,\leq\,(\rs+\delta)n}}
\mathcal{A}_{\ell,r}
\qquad\text{where}\quad
\mathcal{A}_{\ell,r}
\ =\ 
\big\{L_n\geq \ell,\,R_n\geq r\big\}
\,.
\]
Recalling now that the event on the left-hand side has probability at
least~$1/2$,
by a simple union bound we can
find~$\ell,\,r\in\{0,\,\ldots,\,n\}$ such
that~$n-\ell-r\leq(\rs+\delta)n$ and
\begin{equation}
\label{ProbaAlr}
\PP(\mathcal{A}_{\ell,r})
\ \geq\ 
\frac 1{2(n+1)^2}
\,.
\end{equation}

We now consider the following three-steps procedure to
stabilize~$V_n$ starting with one active particle per site:

\begin{enumerate}
\item During the first step, we always topple the leftmost active
particle in~$V_n$, until either~$\ell$ particles
jumped out by the left exit, or no active particle remains.

\item If the configuration is already stable after step~1 or if already
at least~$r$ particles jumped by the right exit during step~1, then we
do nothing during step~2.
Otherwise, step~2 consists in always toppling the rightmost active
particle in~$V_n$ until either no active particle remains, or a total
of~$r$ particles have jumped out by
the right exit during steps~1 and~2.

\item If the configuration is still not stable after step~2, during
the third and last step we perform legal topplings until the configuration
is stable in~$V_n$.
\end{enumerate}

On the event~$\mathcal{A}_{\ell,r}$, step~1 cannot stop before~$\ell$
particles jump out by the left exit.
This implies that the last toppling performed during step~1 is a
toppling on the leftmost site of~$V_n$, which entails that this site
contained an active particle just before the end of step~1.
Yet, note that during step~1 there can never be a sleeping particle
to the left of an active particle.
Hence on the event~$\mathcal{A}_{\ell,r}$ step~1 terminates with no sleeping
particle in~$V_n$ and with~$\ell$ particle having jumped by the left
exit.

For the same reasons, on the event~$\mathcal{A}_{\ell,r}$ step~2 necessarily
terminates with at least~$r$ particles having jumped out by the right
exit (in total during steps~1 and~2), and with no sleeping particles
in~$V_n$.

To sum up, on the event~$\mathcal{A}_{\ell,r}$, step~2 always
terminates with at most~$n-\ell-r\leq(\rs+\delta)n$ particles
remaining in the segment, and these particles are all active.
Thus, by definition of~$\varepsilon$, conditioned on the event~$\mathcal{A}_{\ell,r}\cap\{\text{at
least }(\rs-\delta)n\text{ particles remain after step~2}\}$, with
probability at least~$c$ at least~$\varepsilon n$ particles jump out during
step~3, leaving at most~$(\rs+\delta)n-\varepsilon n = (\rs-\delta)n$
particles at the end of
step~3.
Hence, we obtain
\[
\PP\bigg(\frac{S_n}n\leq\rs-\delta\bigg)
\ \geq\ 
c\,\PP\big(\mathcal{A}_{\ell,r}\big)
\ \geq\ 
\frac{c}{2(n+1)^2}
\,,
\]
using the bound~\eqref{ProbaAlr}.
This being true for every~$n$ large enough, we obtain a contradiction
with the fact that
this probability decays exponentially fast with~$n$ (as ensured by the
superadditivity property and its consequence, Lemma~\ref{lemma-cv}).
The proof by contradiction that~$\rs\leq\rc$ is thereby complete.

Note that the above proof takes advantage of the exponential
bound in order to compensate the entropic factor~$(n+1)^2$ coming from the
number of choices of~$\ell$ and~$r$.
Yet, it is also possible to obtain the result using only the convergence in probability
of~$S_n/n$.
To do so, impose~$\ell$ and~$r$ to be
multiples of~$\delta n$ and take~$r$ as a function of~$\ell$.
This adds an error of at
most~$\delta n$ (which we can afford by taking more margin, choosing
for example~$\delta=\varepsilon/3$
instead of~$\varepsilon/2$), and reduces the
number of possible couples~$(\ell,\,r)$ to a constant of the
order~$1/\delta^2$.
In Section~\ref{sec-ps}, we will use this entropy reduction technique
to obtain the ball conjecture.

\section{The hockey stick: proof of Theorem~\ref{thm-hockey}}
\label{sec-hockey}

We now prove the hockey stick conjecture, which in fact easily follows
from what we just did in Section~\ref{sec-dd}.

Note that, following the Abelian property, for every~$t\in\N$ the
variable~$Y_t$ is the number of particles which remain in~$V_n$ after
stabilization starting with an initial configuration consisting of~$t$
active particles placed independently and uniformly in~$V_n$.

\subsection{Upper bound for subcritical densities}

The upper bound for subcritical densities is trivial because by
definition of~$Y_t$ we always have~$Y_t\leq t$.

\subsection{Upper bound for supercritical densities}

Note that for every~$t\in\N$, it follows from the monotonicity
property given by Lemma~\ref{lemma-mon-active} that~$Y_t$ is
stochastically dominated by~$S_n$ (because starting from any initial
configuration with only active particles, one may first topple sites
with at least two particles until there are no more such sites, and
then one is left with a configuration with on each site,~$0$ or~$1$
active particle).

Therefore, for every~$\rho>\rc$ and for
every~$\varepsilon>0$ we have
\[
\PP\bigg(\frac{Y_{\lceil\rho n\rceil}}n>\rc+\varepsilon\bigg)
\ \leq\ 
\PP\bigg(\frac{S_n}n>\rc+\varepsilon\bigg)
\,,
\]
which tends to~$0$ as~$n\to\infty$ by Theorem~\ref{thm-dd}.

\subsection{Lower bound}

The lower bound follows the same line of proof as the lower bound
on~$S_n$ established in Section~\ref{sec-dd-lower}.
Let~\smash{$\rho>0$} and let~$\rho'<\min(\rho,\,\rc)$.
Let~$n\geq 1$ and let~$\eta$ be a random initial configuration
with~$\lceil\rho n\rceil$ active particles placed independently and
uniformly in~$V_n$.
Let~$\rho''$ be such that~$\rho'<\rho''<\min(\rho,\,\rc)$.
Then, one may couple this configuration~$\eta$ with an i.i.d.\
configuration~$\eta'$ with mean~$\rho''$ such
that~$\eta\geq\eta'$ with
probability tending to~$1$ as~$n\to\infty$.
The result then follows using Lemma~\ref{lemma-mon-active}
(monotonicity) and Lemma~\ref{lemma-RT} (starting with an i.i.d.\
subcritical configuration, the probability that a positive fraction of
the particles jumps out tends to~$0$).

\section{Growth of a ball: proof of Theorem~\ref{thm-ps}}
\label{sec-ps}

The proof of Theorem~\ref{thm-ps} is divided into an inner bound
(showing that it is unlikely that the aggregate is so small that it
hosts a supercritical
density of particles) and an outer bound (showing that it is unlikely
that the aggregate is so spread out that the density of particles
inside is subcritical).
The inner bound relies on the fact that for every~$\rho>\rho_c$ we
have~$\PP(S_n/n\geq\rho)\to 0$ as~$n\to\infty$ (according to
Theorem~\ref{thm-dd}), while the outer bound
relies on Lemma~\ref{lemma-RT}, along with Lemma~\ref{lemma-NML}.
Since these results do not give estimates on the speed at which the
probabilities tend to~$0$, in both the inner bound and the outer bound
we use ``entropy reduction tricks'' to avoid summing a too large
number of terms.

\subsection{Inner bound}

We start by proving the following result:

\begin{lemma}
\label{lemma-inner}
For every~$n,\,k\geq 1$, for every~$x\in V_n$, we have
\[
\PP(x+A_k\subset V_n)
\ \leq\ 
\PP(S_n\geq k)
\,.
\]
\end{lemma}

\begin{proof}
Let~$n,\,k\geq 1$ and~$x\in V_n$, and consider the initial
configuration~$\eta=\1_{V_n}+k\delta_x$, with~$k+1$ active particles
at~$x$ and one active particle at each other site of~$V_n$.

On the one hand, if we first move the~$k$ particles from~$x$ out of~$V_n$,
on top of the other~$n$ particles which we do not move, and then 
let this remaining carpet of one active particle per site stabilize,
then the number of remaining particles is distributed as~$S_n$.

On the other hand, if we start by forcing the~$n$ particles of the
carpet to move out of~$V_n$, using acceptable topplings, and then let
the~$k$ particles at~$x$ stabilize, then with
probability~$\PP(x+A_k\subset V_n)$, these~$k$ particles all remain
inside~$V_n$.

Note that in the first scenario we only used legal topplings, while in the
second scenario we used acceptable topplings.
Thus, the
number of particles remaining inside~$V_n$ being a non-increasing function of
the odometer, the least action principle given by Lemma~\ref{lemma-monotonicity} allows us to conclude.
\end{proof}

With this result in hand we now turn to the proof of the inner
bound.

\begin{proof}[Proof of Theorem~\ref{thm-ps}, inner bound]
Let~$\rho>\rc$, and let us show that~$\PP(k/|A_k|\geq \rho)\to 0$
as~$k\to\infty$.
For every~$k,\,n\geq 1$, we can write
\begin{equation}
\label{entropic-pb}
\big\{|A_k|\leq n\big\}
\ =\ 
\bigcup_{x\in V_n}
\big\{x+A_k\subset V_n\big\}
\,.
\end{equation}
At this point, we could perform a union bound and use
Lemma~\ref{lemma-inner} to deduce that~$\PP(|A_k|\leq n)\leq
n\,\PP(S_n\geq k)$, and take~$n=n_k=\lfloor k/\rho\rfloor$.
But then, we only know from Theorem~\ref{thm-dd} that~$\PP(S_{n_k}\geq
k)$ tends to~$0$ as~$k\to\infty$, but we do not know at which speed,
so we cannot beat the entropic factor~$n$ in front of the probability.

To bypass this issue, we take some margin in order to be able to
reduce the number of terms.
We fix an intermediate density~$\rho'\in(\rc,\,\rho)$.
Then, for every~$k,\,n,\,m\geq 1$,
taking~$y=\lceil x/m\rceil$ in~\eqref{entropic-pb} we get
\[
\big\{|A_k|\leq n\big\}
\
\ \subset\ 
\bigcup_{y=1}^{\lceil n/m\rceil}
\big\{ym+A_k\subset V_{n+m}\big\}
\,.
\]
Combining this with Lemma~\ref{lemma-inner}, we deduce that
\[
\PP\big(|A_k|\leq n\big)
\ \leq\ 
\sum_{y=1}^{\lceil n/m\rceil}
\PP\big(ym+A_k\subset V_{n+m}\big)
\ \leq\ 
\Big\lceil\frac n m\Big\rceil
\,\PP(S_{n+m}\geq k)
\,.
\]
Taking now~$n=n_k=\lfloor k/\rho\rfloor$ and~$m=m_k=\lfloor
n_k(\rho-\rho')/\rho'\rfloor$ we can write
\[
\PP\bigg(\frac k {|A_k|} \geq\rho\bigg)
\ =\ 
\PP\big(|A_k|\leq n_k\big)
\ \leq\ 
\Big\lceil\frac {n_k}{m_k}\Big\rceil
\,\PP\bigg(\frac{S_{n_k+m_k}}{n_k+m_k}\geq\rho'\bigg)
\,,
\]
using that~$k\geq n_k\rho\geq (n_k+m_k)\rho'$.
When~$k\to\infty$ we have~$n_k+m_k\to\infty$, so that the last probability
above tends to~$0$, following the convergence in
probability~$S_n/n\to\rc$ established in Theorem~\ref{thm-dd}.
Besides, we have~$n_k/m_k=O(1)$ when~$k\to\infty$, so we eventually
deduce that~$\PP(k/|A_k|\geq \rho)\to 0$ when~$k\to\infty$, which is
the desired inner bound.
\end{proof}

\subsection{Outer bound}
\label{sec-outer}

Let~$\rho<\rc$.
As for the outer bound in~\cite{LS21}, we proceed in two steps: first
we make the particles spread (using acceptable topplings) to obtain a
subcritical particle density, and then we let the particles stabilize,
using that at subcritical densities few particles jump out of a
segment.
Thanks to Lemma~\ref{lemma-NML} this implies that, with high
probability, no particle jumps out of a slightly enlarged segment.

Let~$\rho',\,\rho''$ be such that~$\rho<\rho'<\rho''<\rc$.
Let~$k\geq 1$.
Let~$\eta:\Z\to\{0,1\}$ be an i.i.d.\ Bernoulli initial configuration
with parameter~$\rho''$.
We see the sites~$x$ such that~$\eta(x)=1$ as holes, which can be
filled by one particle.

We start with~$k$ active particles at~$0$ and no other particles elsewhere.
During the first step, we force each of these~$k$ particles to walk,
with acceptable topplings, until it finds an unoccupied hole.
At the end of this first step, we end up with the~$k$ particles placed
in~$k$ consecutive holes.
Let~$I$ be the set of sites visited during this first step.
Note that this set~$I$ contains exactly~$k$ holes, i.e., we always
have~$\|\eta\|_I=k$, and the configuration after step~1 is
simply~$\eta\1_I$.

Then, in the second step we simply leave these~$k$ particles stabilize
in~$\Z$ with legal topplings.
Let~$B_k$ be the set of sites of~$\Z$ which are visited at least once during
these two steps.
The least action principle indicated by
Lemma~\ref{lemma-monotonicity} entails that~$A_k\subset B_k$.
Therefore, we aim at showing the outer bound for~$B_k$ instead
of~$A_k$, i.e., we show that~$\PP(k/|B_k|\leq\rho)\to 0$.

First, taking~$n=\lceil k/\rho'\rceil$, we can write
\[
\big\{|I|\geq n\big\}
\ =\ 
\bigcup_{x\in V_n}\big\{I\supset V_n-x\big\}
\ \subset\ 
\bigcup_{x\in V_n}
\big\{
k=\|\eta\|_I\geq\|\eta\|_{V_n-x}
\big\}
\,,
\]
so that
\[
\PP\big(|I|\geq n\big)
\ \leq\ 
\sum_{x\in V_n}\PP\big(\|\eta\|_{V_n-x}\leq k\big)
\ =\ 
n\,\PP\big(\|\eta\|_{V_n}\leq k\big)
\ =\ 
n\,\PP\bigg(\frac{\eta(1)+\cdots+\eta(n)}{n}\leq\rho'\bigg)
\,,
\]
which tends to~$0$ when~$n\to\infty$ by the law of large numbers,
since~$\rho'<\rho''=\E[\eta(1)]$.

Thus, to get the desired outer bound, there only remains to prove
that~$\PP(\mathcal{E}_k)\to 0$ as~$k\to\infty$,
where
\[
\mathcal{E}_k
\ =\ 
\bigg\{|B_k|\geq \frac k{\rho},\,|I|<n_k\bigg\}
\qquad
\text{with}
\quad
n_k
\ =\ 
\bigg\lceil\frac k{\rho'}\bigg\rceil
\,.
\]
Now, we would like to divide this event according to the position of
the set~$I$, which has to be included into a segment of length~$n_k$
containing~$0$.
But the number of such segments is~$n_k$, which would introduce an
entropic factor that we would not be able to deal with, because in the
end we will use the result of Lemma~\ref{lemma-RT} which says that
below criticality the probability that a positive fraction exits tends
to~$0$, but does not tell us at which speed.

Hence, as we did above for the inner bound, we circumvent this problem
by taking some margin to reduce the number of events in the union.
For~$n=n_k$ and~$m\geq 1$ we write
\begin{equation}
\label{union-Ek}
\mathcal{E}_k
\ \subset\ 
\bigcup_{x\in V_n}
\bigg\{|B_k|\geq\frac k\rho\,,\ I\subset V_n-x\bigg\}
\ \subset\ 
\bigcup_{y=1}^{\lceil n/m\rceil}
\mathcal{F}_{k,y}
\,,
\end{equation}
where the events~$\mathcal{F}_{k,y}$ are defined by
\[
\mathcal{F}_{k,y}
\ =\ 
\bigg\{|B_k|\geq\frac k\rho\,,\ I\subset V_{n+m}-ym\bigg\}
\,.
\]
We now choose~$\alpha,\,\beta>0$ such that
\[
\frac 1{\rho'}+\alpha+4\beta
\ <\ \frac 1\rho
\]
and we define~$m=m_k=\lceil \alpha k\rceil$ and~$j=j_k=\lceil \beta
k\rceil$ for each~$k\geq 1$, so that~$n+m+4j<k/\rho$ for~$k$ large enough.
Let~$k$ be large enough so that this holds, and let~$y\leq\lceil
n/m\rceil$ and~$J_y=V_{n+m}-ym$.
Defining~\smash{$K_y=\{1-ym-2j,\,\ldots,\,n+m-ym+2j\}$} we
have~$|K_y|=n+m+4j<k/\rho$, so that
\[
\mathcal{F}_{k,y}
\ \subset\ 
\big\{B_k\not\subset K_y\,,\ I\subset J_y\big\}
\,.
\]
Now recall the notation~$A(\eta)$ for the set of sites visited during
the stabilization of a configuration~$\eta$ in~$\Z$.
With this notation, we have~$B_k=I\cup A(\eta\1_I)$.
Since~$J_y\subset K_y$, we get
\[
\mathcal{F}_{k,y}
\ \subset\ 
\big\{A(\eta\1_I)\not\subset K_y\,,\ I\subset J_y\big\}
\ \subset\ 
\big\{A(\eta\1_{J_y})\not\subset K_y\big\}
\,.
\]
Thus, we obtain
\[
\PP(\mathcal{F}_{k,y})
\ \leq\ 
\PP\big(A(\eta\1_{J_y})\not\subset K_y\big)
\ =\ 
\PP\big(A(\eta\1_{V_{n+m}})\not\subset K_0\big)
\,.
\]
Using now Lemma~\ref{lemma-NML} we deduce that, for
every~$i\in\N$,
\[
\PP(\mathcal{F}_{k,y})
\ \leq\ 
1-\PP_{\eta}(M_{n+m}\leq i)
\,\PP(G_1+\cdots G_i\leq j)
\,,
\]
where~$G_1,\,\ldots,\,G_i$ are i.i.d.\ Geometric
variables with parameter~$\lambda/(1+\lambda)$.
Plugging this into~\eqref{union-Ek}, we are left with
\[
\PP(\mathcal{E}_k)
\ \leq\ 
\Big\lceil\frac n m\Big\rceil\,
\Big(
1-\PP_\eta(M_{n+m}\leq i)
\,\PP(G_1+\cdots G_i\leq j)
\Big)
\,.
\]
Choosing now~$i=i_k=\lfloor\gamma k\rfloor$ with a certain
parameter~$\gamma>0$ such that~$\gamma\lambda/(1+\lambda)<\beta$, we
then have~$\PP(G_1+\cdots G_i\leq j)\to 1$ when~$k\to\infty$ by the
law of large numbers, and~$\PP(M_{n+m}\leq i)\to 1$ by
Lemma~\ref{lemma-RT}.
Since~$n_k/m_k=O(1)$ when~$k\to\infty$, we conclude
that~$\PP(\mathcal{E}_k)\to 0$ as~$k\to\infty$, which completes the
proof of the outer bound.

\bibliography{article.bib}
\end{document}